\documentclass[a4paper]{scrartcl}
\usepackage{amssymb, amsmath, amsthm, mathrsfs, pxfonts, braket}

\usepackage{lmodern}       

\usepackage{graphicx}

\usepackage{hyperref}
\usepackage[hyperpageref]{backref} 
\usepackage{enumitem}

\usepackage[all,2cell]{xy}
\objectmargin+{0.5mm}
\labelmargin+{0.8mm}
\SelectTips{cm}{12}

  \makeatletter
    
    \@addtoreset{equation}{section}
  \makeatother

\newtheorem{theorem}{Theorem}[section]
\newtheorem{lemma}[theorem]{Lemma}

\theoremstyle{definition}
\newtheorem{remark}[theorem]{Remark}

\setlist[enumerate]{itemsep=0pt, label=$(\mathrm{\roman*})$, topsep=1pt}
\setlist[description]{itemsep=0pt, topsep=1pt}
\setlist[itemize]{labelindent=\parindent,leftmargin=*,align=left, itemsep=0pt,topsep=2pt}

\usepackage{etoolbox}
\usepackage{framed}
\newcommand\enclosebox[2]{%
  \BeforeBeginEnvironment{#1}{\begin{#2}}%
  \AfterEndEnvironment{#1}{\end{#2}}%
}

\enclosebox{theorem}{leftbar}
\enclosebox{lemma}{leftbar}
\enclosebox{definition}{leftbar}
\enclosebox{remark}{leftbar}
\enclosebox{example}{leftbar}

\newcommand{\Aut}{\operatorname{Aut}}

\newcommand{\C}{\mathbb{C}}
\newcommand{\Cf}{\textit{cf.}\;}

\newcommand{\Cl}{\operatorname{Cl}}

\newcommand{\Ebar}{\ol{E}}
\newcommand{\Ebarns}{\Ebar_{\ns}}
\newcommand{\Ehat}{\wh{E}}

\newcommand{\F}{\mathbb{F}}

\newcommand{\Fp}{\F_p}
\newcommand{\Fl}{\F_l}

\newcommand{\Gal}{\operatorname{Gal}}

\newcommand{\Ie}{\textit{i.e.},\ }

\newcommand{\K}{\mathcal{K}}

\renewcommand{\L}{\mathfrak{L}}
\renewcommand{\l}{\mathfrak{l}} 

\newcommand{\ns}{\mathrm{ns}}
\renewcommand{\O}{\mathcal{O}}

\newcommand{\onto}[1]{\stackrel{#1}{\to}}
\newcommand{\ol}[1]{\overline{#1}}
\newcommand{\ord}{\operatorname{ord}}

\newcommand{\p}{\mathfrak{p}}
\renewcommand{\P}{\mathfrak{P}}

\newcommand{\Phii}{\Phi^{(i)}}

\newcommand{\Q}{\mathbb{Q}}
\newcommand{\Qbar}{\ol{\Q} }

\newcommand{\Qp}{\mathbb{Q}_{p}}

\newcommand{\Ql}{\Q_l}

\newcommand{\tor}{\mathrm{tor}}

\newcommand{\wt}[1]{\widetilde{#1}}
\newcommand{\wh}[1]{\widehat{#1}}

\newcommand{\Z}{\mathbb{Z}}
\newcommand{\Zp}{\mathbb{Z}_p}
\newcommand{\Zl}{\mathbb{Z}_l}

\newcommand\sn{\smallskip\noindent}

\newcommand{\Redp}{$(\mathbf{Red}_p)$}
\newcommand{\Redl}{$(\mathbf{Red}_l)$}

\newcommand{\Tor}{$(\mathbf{Tor})$}
\newcommand{\Full}{$(\mathbf{Full})$}
\newcommand{\Disc}{$(\mathbf{Disc})$}

\title{Local torsion primes and the class numbers associated to an elliptic curve over $\Q$}
\author{Toshiro Hiranouchi}

\begin{document}
\pagenumbering{arabic}
\maketitle

\begin{abstract}
Using the rank of the Mordell-Weil group $E(\Q)$ 
of an elliptic curve $E$ over $\Q$,  
we give a lower bound   
of the class number 
of the number field $\Q(E[p^n])$ 
generated by $p^n$-division points of $E$ 
when the curve $E$ does not possess a $p$-adic point of order $p$\,: $E(\Qp)[p] =0$.  

\sn
\end{abstract}

\section{Introduction}
For an elliptic curve $E$ over $\Q$ with complex multiplication 
(abbreviated as CM in the following) such that 
$\operatorname{End}_{\C}(E) = \O_F$, 
the ring of integers of an imaginary quadratic field $F$. 
When $E$ has good ordinary reduction at $p>2$,  
the prime $p$ splits completely in $F$ as $p = \pi \ol{\pi}$ where $\pi \in \O_F$ 
and $\ol{\pi}$ is the complex conjugation of $\pi$. 
Let $F_n := F(E[\pi^{n}])$ be the field generated by $\pi^n$-torsion points of $E$ over $F$. 
The extension 
$F_{\infty} := \bigcup_n F_n$ of $F_1$ is a $\Zp$-extension 
so that 
there exist $\lambda,\mu \in \Z_{\ge 0}$ and $\nu \in \Z$ which are 
all independent of $n$ such that 
we have 
\[
	\# \Cl_p(F_n) = p^{\lambda n + \mu p^n + \nu},\quad  \mbox{for $n \gg 0$}, 
\]
where $\Cl_p(F_n)$ is the $p$-Sylow subgroup of the ideal class group of $F_n$. 
It is known that 
the invariant $\lambda$ of the $\Zp$-extension has a lower bound 
\[
	\lambda \ge r-1,
\]
where $r$ is the ($\Z$-)rank of the group of $\Q$-rational points $E(\Q)$ 
(\cite{G}, Sect.~5). 

For an elliptic curve $E$ over $\Q$ which may not have CM and a prime number $p>3$, 
in recent papers \cite{SY15} and \cite{SY},  
Sairaiji and Yamauchi gave 
a lower bound of the class number $\# \Cl_p(K_n)$ 
in terms of the rank of $E(\Q)$ 
associated to the field $K_n := \Q(E[p^n])$ 
generated by $p^n$-torsion points $E[p^n] := E(\Qbar)[p^n]$ 
under the following conditions\footnote{
In \cite{SY15}, the cases $p=2$ and $3$ have been studied 
under the additional condition:  
 $\Gal(K_n/\Q) \simeq GL_2(\Z/p^n\Z)$ for all $n\ge 1$. 
 In fact, for $p>3$, \Full\ implies this condition (\Cf \cite{SY}, Sect.~1).
}: 

\begin{itemize}
	\item[\Redl] $E$ has multiplicative reduction or potentailly good reduction at any prime $l\neq p$, 
	\item[\Redp] $E$ has multiplicative reduction at $p$, 
	\item[\Disc] $p \nmid \ord_p(\Delta)$, where $\Delta$ is the minimal discriminant of $E$, and 	
	\item[\Full] $\Gal(K_1/\Q) \simeq GL_2(\Z/p\Z)$.
\end{itemize}

\sn 
When $p>5$ the condition \Redp\ automatically implies \Disc\ 
(\Cf \cite{SY}, Sect.~1). 
The objective of this note is to propose a condition 

\begin{itemize}
	\item[\Tor] $E(\Qp)[p] = 0$ 
\end{itemize}

\noindent
instead of using \Redp\ and \Disc\ above, 
and give the same form of a lower bound of $\# \Cl_p(K_n)$ as in \cite{SY}. 
The main theorem is the following:

\begin{theorem}[\Cf Thm.~\ref{thm:main}]
\label{thm:main_intro}
Let $E$ be an elliptic curve over $\Q$ 
with minimal discriminant $\Delta$ 
and let $p$ be a prime number $>3$. 
Put $K_n := \Q(E[p^n])$. 
Assume the conditions \Tor\ and \Full\ noted above. 
Then, for all $n\in \Z_{\ge 1}$, 
the exponent $\kappa_n$ of $\#\Cl_p(K_n) = p^{\kappa_n}$ satisfies the following inequality: 
\[
	\kappa_n \ge 2n(r - 1) - 2 \sum_{l\neq p,\, l\mid \Delta} \nu_l, 
\] 
where 
$r$ is the rank of $E(\Q)$ and 
\[
	\nu_l :=
	\begin{cases}
		\min\set{\ord_p(\ord_l(\Delta)),n}, & \mbox{if $E$ has split multiplicative reduction at $l$},\\
		0, & \mbox{otherwise}.
	\end{cases}
\]
Here, $\ord_p$ (resp.\ $\ord_l$) is the $p$-adic (resp.\ $l$-adic) valuation on $\Q$. 
\end{theorem}

\begin{remark}
\label{rem:cond}
\begin{enumerate}
	\item 
	The condition \Full\ means that 
	the Galois representation 
	$\rho:\Gal(\Qbar/\Q) \to \Aut(E[p]) \simeq GL_2(\Z/p\Z)$ 
	is full (i.e., surjective). 
	This can be checked by some criterions    
	\cite{S}, Sect.~2.8   
	(see also \cite{SY}, Sect.~1).   
	
	\item 
	In \cite{DW}, 
	for an elliptic curve $E$ over $\Q$, 
	a prime number $p$ which does not 
	satisfy \Tor, that is, $E(\Qp)[p]\neq 0$ 
	is called a \textbf{local torsion prime} for $E$. 
	It is expected that 
	when $E$ does not have CM, 
	there are only finitely many local torsion primes $($\cite{DW}, Conj.~1.1$)$. 
	
%
	\end{enumerate}
\end{remark}

%
A proof of Theorem~\ref{thm:main_intro} is given in Section~\ref{sec:global} (\Cf Thm.~\ref{thm:main}). 
In Section~\ref{sec:tor},  
we give some sufficient conditions 
for \Tor. 
In fact, the conditions \Redp\ and \Disc\ imply the condition \Tor\ (Lem.~\ref{lem:Delta}). 
Even though the theorem above can be applied to 
an elliptic curve and a prime $p$ of a wider class 
than \cite{SY}, 
the proof is significantly simplified.

%
%

Closing this section, 
let us consider the elliptic curve $E$ over $\Q$ defined by 
\[
	y^2+y=x^3+x^2-2x
\] 
(the Cremona label 389a1) 
which has the smallest conductor among those of $r=2$. 
This $E$ does not have CM and 
$\Delta = 389$ 
($E$ has multiplicative reduction at $389$). 
By using SAGE \cite{sage}, 
one can confirm that 
the condition \Full\ holds for all primes $p$ 
and \Tor\ holds for any odd prime $<10^6$. 
Thus, our main theorem says that, 
for all odd primes $p <10^6$ 
(which may be $p = 389$),  
we have 
\[
	\# \Cl_p(K_n) \ge p^{2n(r-1)}. 
\]

\subsection*{Acknowledgement}
The author would like to thank 
Professor F.~Sairaiji and Professor T.~Yamauchi 
who taught the author their results in \cite{SY15} and \cite{SY}. 
Not only they generously sent the author their preprint \cite{SY}, but also 
gave suggestions and comments which are improved the main theorem in this note. 
The author would like to thank Professor K.~Matsuno for pointing out an error of 
the proof of Lemma \ref{lem:Delta} in an early draft of this note. 
The arguments in the latter part of Lemma~\ref{lem:local} are due to them. 
The author would like to thank the referee for some comments which amend 
this note.

\section{Local torsion primes}
\label{sec:tor}

Throughout this note, 
we use the following notation: 

\begin{itemize} 
\item $p$ : a prime number $>2$, 
\item $E$ : an elliptic curve over $\Q$,
\item $\Delta$ : the minimal discriminant of $E$ (\cite{106}, Chap.\ VIII, Sect.\ 8), 
\item $[p^n]:E\to E$ : the isogeny multiplication by $p^n$ (\cite{106}, Chap.~III, Sect.~4), and
\item $E[p^n] := E(\Qbar)[p^n]$ : the $p^n$-torsion subgroup of $E(\Qbar)$.
\end{itemize}

%

\subsection*{Structure theorem on $E(\Ql)$}
For a second prime number $l$ (which may be $p$), 
we denote also by $E$  
the base change $E\otimes_{\Q} \Ql$ of the elliptic curve $E$ to $\Ql$. 
Define 
\begin{itemize}
	\item  $\pi:E(\Ql) \to \Ebar(\Fl)$\,: the reduction map 
	modulo $l$ (\cite{106}, Chap.~VII, Sect.~2), 
	\item  $\Ebarns(\Fl)$: the set of non-singular points in 
the reduction $\Ebar(\Fl)$
(\Cf \cite{106},  Chap.~III, Prop.~2.5), and 
	\item $E_0(\Ql) := \pi^{-1}(\Ebar_{\ns}(\Fl))$. 
\end{itemize}
The reduction map $\pi:E(\Ql)\to \Ebar(\Fl)$ modulo $l$ 
	induces a short exact sequence (of abelian groups)  
	\begin{equation}
		\label{eq:red_l}
		0\to E_1(\Ql) \to E_0(\Ql) \onto{\pi} \Ebar_{\ns}(\Fl) \to 0, 
	\end{equation}
	where $E_1(\Ql)$ is defined by the exactness 
	(\Cf \cite{106}, Chap.~VII, Prop.~2.1). 

\begin{lemma}
\label{lem:str}
\begin{enumerate}
	\item $E_1(\Ql) \simeq \Zl$. 
	In particular, $E_1(\Ql)[p] = 0$. 
	\item 
	\begin{itemize}
	\item [$(\mathrm{a})$] 
	If $E$ has multiplicative reduction at $l$, 
	then $\Ebarns (\Fl) \subset \Ebarns(\F_{l^2}) \simeq (\F_{l^2})^{\times}$. 
	\item [$(\mathrm{b})$] 
	If $E$ has additive reduction at $l$,  
	then $\Ebarns(\Fl)  \simeq \Fl$ as additive groups. 
	\end{itemize}
	
	\item 
	\begin{itemize}
	\item[$(\mathrm{a})$] 	
	If $E$ has split multiplicative reduction at $l$, 
	then $E(\Ql)/E_0(\Ql) \simeq \Z/\ord_l(\Delta)\Z$.
	\item [$(\mathrm{b})$] 
	If $E$ has non-split multiplicative reduction at $l$, 
	then $E(\Ql)/E_0(\Ql)$ is a finite group of order at most $2$. 
	\item [$(\mathrm{c})$]
	In all other cases, namely, 
	$E$ has good reduction or additive reduction at $l$, 
	the quotient $E(\Ql)/E_0(\Ql)$ is 
	a finite group of order at most $4$. 
	\end{itemize}
	\item 
	We have an isomorphism 
\[
	E(\Ql) \simeq \Zl \oplus E(\Ql)_{\tor}
\]
as abelian groups, 
where $E(\Ql)_{\tor}$ is the torsion subgroup of $E(\Ql)$ 
which is finite.
\end{enumerate}	
\end{lemma}
\begin{proof}
	(i) We have 
	$E_1(\Ql) \simeq \Ehat(l\Zl)$, 
	where $\Ehat(l\Zl)$ is the 
	group associated to the formal group $\Ehat$ of $E$ 
	(\cite{106}, Chap.~VII, Prop.~2.2). 
	From \cite{106}, Chapter~IV, Theorem 6.4 (b), 
	the formal logarithm induces $\Ehat(l\Zl) \simeq l\Zl \simeq \Zl$. 
	As $\Zl$ is $p$-torsion free for $p>2$, 
	we obtain $E_1(\Ql)[p]=0$. 
	
\sn
	(ii) The assertion follows from \cite{106}, Chapter~III, Exercise~3.5. 
%

\sn
	(iii) 
	 The assertion follows from 
	\cite{106}, Chapter~VII, Theorem~6.1 for the case (a) and (c), 
 	and   
 	\cite{151}, Chapter~IV, Remark~9.6 for the case (b). 
 	
\sn
	(iv)
	As the exact sequence \eqref{eq:red_l} splits, 
	we have $E_0(\Ql) \simeq E_1(\Ql) \oplus \Ebarns(\Fl)$. 
	From (i), $E_1(\Ql) \simeq \Zl$ and the quotients 
	$E(\Ql)/E_0(\Ql), E_0(\Ql)/E_1(\Ql) \simeq \Ebarns(\Fl)$ 
	are finite by (ii) and (iii). 
	The assertion follows from this.   
\end{proof}

Recall that 
the \textbf{Tamagawa number} $c_l$ at a prime $l$ 
for $E$ is defined by
\begin{equation}
\label{def:Tam}
	c_l := (E(\Ql):E_0(\Ql)). 
\end{equation}

\begin{lemma}
\label{lem:add}
	Suppose that $E$ has additive reduction at a prime $l\neq p$. 
		We further assume the following conditions: 
		\begin{itemize}
		\item[$(\mathrm{a})$] $p>3$,  or
		\item[$(\mathrm{b})$] $c_l \neq 3$, where $c_l$ is the Tamagawa number at $l$ $($\Cf \eqref{def:Tam}$)$. 
		\end{itemize}
	Then, $E(\Ql)[p] = 0$.
\end{lemma}
\begin{proof}
	As $E$ has additive reduction at $l$, 
	we have  $\Ebarns(\Fl)[p] =0$ (Lem.~\ref{lem:str} (ii-b)). 
	On the other hand, $E_1(\Ql)[p] = 0$ (Lem.~\ref{lem:str} (i))
	so that 
	$E_0(\Ql)[p] = 0$ by \eqref{eq:red_l}. 
	As $c_l = \#E(\Ql)/E_0(\Ql)\le 4$ (Lem.~\ref{lem:str} (iii)), 
    the quotient 
    $E(\Ql)/E_0(\Ql)$ does not possess elements of order $p$ 
    under the additional assumption (a) or (b).
    We obtain $E(\Ql)[p]  = 0$.       
\end{proof}

\subsection*{Multiplicative reduction at $p$}
\begin{lemma}
\label{lem:Delta}
Suppose the condition \Redp\ in Introduction, 
that is, $E$ has multiplicative reduction at $p$. 
We further assume one of the following conditions: 

\begin{itemize}
	\item[\Disc] $p \nmid \ord_p(\Delta)$, or
	\item[$(\mathrm{a})$] $E$ has non-split multiplicative reduction at $p$.
\end{itemize}
	Then, the condition 
	\Tor $:E(\Qp)[p] = 0$ 	holds. 
\end{lemma}
\begin{proof}
	As $E$ has multiplicative reduction at $p$, 
	$\Ebarns(\F_p) \subset \Ebarns(\F_{p^2}) \simeq (\F_{p^2})^{\times}$  
	(Lem.~\ref{lem:str} (ii)). 
	In particular,  $\Ebarns(\F_p)[p] =0$. 
	On the other hand, $E_1(\Qp)[p] = 0$ (Lem.~\ref{lem:str} (i)) 
	and hence $E_0(\Qp)[p] = 0$ by \eqref{eq:red_l}. 
	
	\sn
    \textbf{Case (a):} 
    First, we suppose that $E$ has non-split multiplicative reduction. 
    In this case, the quotient group 
    $E(\Qp)/E_0(\Qp)$ is a finite group of order at most $2$ (Lem.~\ref{lem:str} (iii)) 
    so that we obtain $E(\Qp)[p]  = 0$. 
    
    \sn 
    \textbf{Case\,(\text{Disc}):} 
    Next, we assume $p\nmid \ord_p(\Delta)$. 
    From Case (a) above, we may assume that 
    $E$ has split multiplicative reduction at $p$. 
    The assertion follows from 
    $E(\Qp)/E_0(\Qp) \simeq \Z/\ord_p(\Delta)\Z$ 
    (Lem.~\ref{lem:str} (iii)). 
%
\end{proof}

\begin{remark}
When the elliptic curve $E$ over $\Q$ 
has multiplicative reduction at $2$, 
by considering the isomorphism $E(\K) \simeq {\K}^{\times}/q^{\Z}$ 
for some unramified extension $\K/\Q_2$ locally,  
$-1\in {\K}^{\times}$ gives a $2$-torsion element in $E(\Q_2)$. 
Thus the condition \Tor\ at $2$ does not hold: $E(\Q_2)[2] \neq 0$.	
\end{remark}

\subsection*{Good reduction at $p$}

\begin{lemma}
\label{lem:DW}
Suppose that $E$ has good reduction at $p$. 

\begin{enumerate}
	\item We further assume one of the following conditions:  
	\begin{enumerate}[label=$(\mathrm{\alph*})$]
		\item $\Ebar(\Fp)[p] = 0$, or
		\item $E(\Q)_{\tor} \neq 0, p\ge 11$. 
	\end{enumerate}	
	Then, the condition \Tor\  holds. 

\item 
Assume that $E$ has CM, and $p\ge 7$. 
Then, \Tor\ holds 
	if and only if 
$\Ebar(\Fp)[p] = 0$.
\end{enumerate}
\end{lemma}
The lemma above essentially follows from \cite{DW}, Proposition~2.1. 
For the sake of completeness, we give a proof. 
\begin{proof}[Proof of Lem.~\ref{lem:DW}]
(i) 
\textbf{Case (a):}\ 
We have $E_1(\Qp)[p]  =  0$ 
(Lem.~\ref{lem:str} (i)). 
The condition can be checked by using the exact sequence 
\[
  0 \to E(\Qp)[p] \onto{\pi} \Ebar(\Fp)[p] \onto{\delta} \wh{E}(p\Zp)/p\wh{E}(p\Zp),  
\]
where 
$\delta$ is the connecting homomorphism. 
The assumption
$\Ebar(\Fp)[p] = 0$  
implies the condition \Tor. 

\noindent
\textbf{Case\ (b):}\ 
Assume $E(\Qp)[p] \neq 0$. 
By \cite{DW}, Proposition~2.1 (1), 
we have $E(\Q)_{\tor} \simeq \Z/p\Z$.  
From the assumption $p\ge 11$, 
this contradicts with Mazur's theorem on 
$E(\Q)_{\tor}$ (\cite{106}, Chap.~VIII, Thm.~7.5).

\noindent
(ii) 
From (i) (the case (a)), 
it is enough to show that 
if $\Ebar(\Fp)[p] \neq 0$, then $E(\Qp)[p] \neq 0$. 
From Hasse's theorem (\cite{106}, Chap.~V, Thm.~1.1) and $p\ge 7$, 
$\# \Ebar(\Fp) = p$. 
We have $a_p(E) := p+1 - \#\Ebar(\Fp) = 1$. 
This implies  
$E(\Qp)[p] \neq 0$ by \cite{DW}, Proposition~2.1 (3) 
under the assumption that $E$ has CM.
\end{proof}

When $E$ has CM, 
Lemma~\ref{lem:DW} (ii) gives a criterion 
for the condition \Tor. 
On the other hand, 
Lemma~\ref{lem:DW} (i) says that, 
for $p\ge 11$, 
\Tor\ does not hold only if 
\begin{enumerate}[label=$(\mathrm{\alph*}')$]
	\item $\Ebar(\Fp)[p] \neq 0$, and 
	\item $E(\Q)_{\tor} = 0$. 
\end{enumerate}
For our purpose, we further impose 

\begin{enumerate}[label=$(\mathrm{\alph*}')$]
  \setcounter{enumi}{2}
	\item $E$ does not have CM, and 
	\item the rank $r>1$. 
\end{enumerate}

\noindent
The following calculations are given by using SAGE \cite{sage}. 
There are $1733$ elliptic curves with conductor $N<10^4$ satisfying ($\mathrm{b'}$)-($\mathrm{d'}$) above. 
%
%
Among them, only $50$ curves have a local torsion prime $p$ in the range $11\le p < 10^6$, \Ie $E(\Qp)[p] \neq 0$ listed below:
\begin{center}
  \begin{tabular}{|c|c|c||c|c|c||c|c|c||c|c|c|}  \hline
    	  &curve & $p$ &   & curve & $p$ & & curve & $p$ &  &  curve & $p$\\ 	   \hline 
$1$ & 1639b1 & $2833$ & $14$ & 4976a1 & $11$ & $27$ & 7497c1 & $13$ & $40$ & 9082a1 & $13$ \\
$2$ & 1957a1 & $163$ & $15$ & 5171a1 & $23$ & $28$ & 7520e1 & $11$ & $41$ & 9149c1 & $23$ \\
$3$ & 2299b1 & $31$ & $16$ & 5736f1 & $11$ & $29$ & 7826d1 & $19$ & $42$ & 9395a1 & $37$ \\
$4$ & 2343c1 & $17$ & $17$ & 5763d1 & $23$ & $30$ & 8025d1 & $43$ & $43$ & 9467a1 & $19$ \\
$5$ & 2541c1 & $197$ & $18$ & 5982h1 & $197$ & $31$ & 8025d2 & $43$ & $44$ & 9510c1 & $103$ \\
$6$ & 2728d1 & $443$ & $19$ & 6334b1 & $11$ & $32$ & 8048f1 & $2593$ & $45$ & 9535a1 & $31$ \\
$7$ & 3220a1 & $41$ & $20$ & 6405c1 & $113$ & $33$ & 8384j1 & $157$ & $46$ & 9706b1 & $367$ \\
$8$ & 3333b1 & $19$ & $21$ & 6792a1 & $97$ & $34$ & 8495a1 & $43$ & $47$ & 9783b1 & $11$ \\
$9$ & 3997a1 & $167$ & $22$ & 6848p1 & $23$ & $35$ & 8551a1 & $293$ & $48$ & 9789f1 & $541$ \\
$10$ & 4024b1 & $47$ & $23$ & 6896e1 & $29$ & $36$ & 8768h1 & $17$ & $49$ & 9797b1 & $19$ \\
$11$ & 4279c1 & $13$ & $24$ & 7152a1 & $79$ & $37$ & 8950m1 & $271$ & $50$ & 9865b1 & $11$ \\
$12$ & 4504b1 & $19$ & $25$ & 7233a1 & $11$ & $38$ & 8974c1 & $1063$ &  &  &  \\
$13$ & 4768a1 & $109$ & $26$ & 7366g1 & $11$ & $39$ & 8988d1 & $37$ &  &  &  \\
\hline 
\end{tabular}

\smallskip
Table1: local torsion primes
\end{center}

\section{Elliptic curve over $\Q$}
\label{sec:global}

We keep the notation of the last section. 
We further define

\begin{itemize} 
\item $K_n := \Q(E[p^n])$ (\Cf \cite{106}, Chap.~VIII, Prop.~1.2\ (d)), 
\item $r :=$ the rank of $E(\Q)$ (which is finite by the Mordell-Weil theorem \cite{106}, Chapter~VIII),
\item $P_1,\ldots , P_r \in E(\Q)$ : generators of the free part of $E(\Q)$, and 
\item $L_n := K_n\!\left([p^n]^{-1}P_1,\ldots, [p^n]^{-1}P_r\right)$. 
\end{itemize}

\noindent
Following \cite{L}, Chapter~V, Section~5, 
for each $1\le i \le r$, 
define 
\begin{equation}
	\label{eq:Phii}
	\Phii: \Gal(L_n/K_n) \to E[p^n]; \sigma \mapsto \sigma(Q_i) - Q_i,
\end{equation}
where $Q_i \in E(\Qbar)$ with $[p^n]Q_i = P_i$. 
Since $E[p^n] \subset E(K_n)$, 
the map $\Phii$ does not depend on the choice of $Q_i$. 
These homomorphisms $(\Phii)_{1\le i \le r}$ induce 
an injective homomorphism  
\begin{equation}
	\label{eq:Phi}
\Phi: \Gal(L_n/K_n) \to E[p^n]^{\oplus r} ; \sigma \mapsto (\Phii(\sigma))_{i}.
\end{equation}
%
From $E[p^n] \simeq (\Z/p^n\Z)^{\oplus 2}$  (\cite{106}, Chap.~III, Cor.~6.4) 
the extension $L_n/K_n$ is an abelian extension with 
$[L_n:K_n] \le p^{2nr}$.

\subsection*{Inertia subgroups}
For any prime number $l$ and a prime ideal $\l$ in (the ring of integers of) $K_n$ above $l$ 
(we write $\l \mid l$ in the following), we denote by 

\begin{itemize}
\item $I_{\l}$ : the inertia subgroup of $\Gal(L_n/K_n)$ at $\l$ 
(for $L_n/K_n$ is abelian, the inertia subgroup $I_{\l}$ is independent 
of a choice of a prime ideal in $L_n$ above $\l$), and 
\item $I_l := \Braket{I_{\l}\, ; \mbox{prime ideal $\l\mid l$ in $K_n$}}$: 
the subgroup of $\Gal(L_n/K_n)$ generated by $I_{\l}$ for all $\l\mid l$.
\end{itemize}

\noindent
For any prime $\l \mid l$ of $K_n$, 
and a prime $\L$ of $L_n$ above $\l$ (we write $\L\mid \l$\,), we denote by 

\begin{itemize}
	\item $(K_n)_{\l}$ : the completion of $K_n$ at $\l$, and 
	\item $(L_n)_{\L}$ : the completion of $L_n$ at $\L$. 
\end{itemize}

\begin{lemma}
\label{lem:local}
We assume the condition \Tor. 
Then, we have $\# I_p \le p^{2n}$.
\end{lemma}
\begin{proof}
By the structure theorem on $E(\Qp)$ (Lem.~\ref{lem:str} (iv)), 
\[
E(\Qp) \simeq \Zp \oplus E(\Qp)_{\tor}.
\]
From the condition \Tor, we have 
$E(\Qp)_{\tor}/[p^n]E(\Qp)_{\tor}  = 0$ 
%
%
and hence 
\[
  E(\Qp)/[p^n]E(\Qp) \simeq \Z/p^n\Z.
\]
Let $\ol{P}\in E(\Qp)/[p^n]E(\Qp)$ 
(the residue class represented by a point $P\in E(\Qp)$) be a generator of 
the cyclic group $E(\Qp)/[p^n]E(\Qp)$ and, for each index $1\le i \le r$, write 
\[
  \ol{P_i} = \ol{a_i}\cdot \ol{P}\quad \mbox{in $E(\Qp)/[p^n]E(\Qp)$} 
\]
for some $\ol{a_i} \in \Z/p^n\Z$\ $(a_i\in \Z)$. 
Take $1\le i\le r$ such that 
\[
\ord_p(a_i) \le \ord_p(a_j)
\] 
for all $1\le j\le r$. 

For any prime $\P \mid p$ of $L_n$,  
we denote by $\p$ the prime in $K_n$ below $\P$.
%
Using the chosen index $i$, we obtain 
\begin{equation}
\label{eq:mono_p}
	(L_n)_{\P} = (K_n)_{\p}\!\left([p^n]^{-1}P_i\right).
\end{equation}
%
%
Put $K_n' := K_n\!\left([p^n]^{-1}P_i\right) \subset L_n$. 
From the equality \eqref{eq:mono_p}, 
the extension $L_n/K_n'$ is unramified 
(at all primes in $K_n'$) above $\p$. 
As the extension $K_n/\Q$ is Galois, 
this extension $L_n/K_n'$ is unramified above $p$. 
Since $I_p \cap \Gal(L_n/K_n') = \set{1}$, 
the restriction $\Phii|_{I_p}: I_p \to E[p^n]$ of $\Phii$ 
defined in \eqref{eq:Phii} is injective and hence
$\#I_p \le p^{2n}$. 
\end{proof}


\begin{lemma}
\label{lem:local_l}
	Let $l$ be a prime number with $l\neq p$.

	\begin{enumerate}
		\item 
		We have $\#I_l \le p^{2n}$. 
		 
		\item 
		Suppose that $E$ has multiplicative reduction at $l$. 
		We have 
		$\#I_l \le p^{2\nu_l}$, 
		where 
\[
	\qquad \nu_l\! :=\!
	\begin{cases}
		\min\set{\ord_p(\ord_l(\Delta)),n},
		\ \mbox{if $E$ has split multiplicative reduction at $l$},\\
		0,\ \mbox{if $E$ has non-split multiplicative reduction at $l$}.
	\end{cases}	
\]
		\item 
		Suppose that $E$ has additive reduction at $l$. 
		We further assume the following conditions: 
		\begin{itemize}
		\item[$(\mathrm{a})$] $p>3$,  or
		\item[$(\mathrm{b})$] $c_l \neq 3$, where $c_l$ is the Tamagawa number at $l$ $($\Cf \eqref{def:Tam}$)$. 
		\end{itemize}		
		Then, we have $\# I_l = 1$.
	\end{enumerate}
%
%
\end{lemma}
\begin{proof}
%
(i) 
	Take any $\l\mid l$ in $K_n$. 
	For a prime $\L\mid \l$ in $L_n$, let 
	$(T_n)_{\L} := ((L_n)_{\L})^{I_{\l}}$ be the inertia field of $\L$ 
	over $(K_n)_{\l}$  
	which is the fixed field of $I_{\l}$ (\Cf \cite{N}, Chap.~II, Def~9.10). 
	Since $l\neq p$ 
	the extension $L_n/K_n$  
	is tamely ramified at any prime $\L\mid \l$ in $L_n$. 
	The inertia subgroup $I_{\l}  = \Gal((L_n)_{\L}/(T_n)_{\L})$ is cyclic 
	(\Cf \cite{N}, Chap.~II, Sect.~9). 
	There exists $1 \le i \le r$ such that 
	\[
	(T_n)_{\L}([p^n]^{-1}P_j) \subset (T_n)_{\L}([p^n]^{-1}P_i)
	\]
	for any $1\le j\le r$. 
	Since 
	$I_{\l}$ does not depend on the choice of $\L\mid \l$ in $L_n$, 
	the index $i$ above can be chosen independent of $\L\mid \l$.  
	We obtain  
	\begin{equation}
		\label{eq:mono}
	(L_n)_{\L}  = (T_n)_{\L}\!\left([p^n]^{-1}P_i\right)
	\end{equation}
for any prime $\L\mid \l$. 

Put $K_n' := K_n\!\left([p^n]^{-1}P_i\right) \subset L_n$.  
The extension $(T_n)_{\L}/(K_n)_{\l}$ of local fields 
is unramified from the very definition of $(T_n)_{\L}$ 
for any prime $\L\mid \l$ in $L_n$.   
Using the equality \eqref{eq:mono} the extension 
\[
(L_n)_{\L} = (T_n)_{\L}([p^n]^{-1}P_i)\ \mbox{of}\ 
(K_n)_{\l}([p^n]^{-1}P_i) 
\]
is also unramified (\cite{N}, Chap.~II, Prop.~7.2). 
This implies that $L_n/K_n'$ is unramified at 
all primes $\L\mid \l$ in $L_n$. 
As the extension $K_n/\Q$ is Galois, 
this extension $L_n/K_n'$ is unramified above $l$. 
Since $I_l \cap \Gal(L_n/K_n') = \set{1}$, 
the restriction $\Phii|_{I_l}: I_l \to E[p^n]$ of $\Phii$ 
defined in \eqref{eq:Phii} is injective and hence
$\#I_l \le p^{2n}$.

\sn
(ii) 
This assertion is \cite{SY}, Theorem 4.1. 

\sn
(iii)
By Lemma \ref{lem:str} (iv), we have 
\[
E(\Ql) \simeq \Zl \oplus E(\Ql)_{\tor}.
\]
From $E(\Ql)[p]=0$ (Lem.~\ref{lem:add}), 
%
%
we have
\[
  E(\Ql)/[p^n]E(\Ql) =0.
\]
Hence,  
$P_i \in [p^n]E(\Ql)$ for each $i$. 
This implies that, for any prime $\l\mid l$ in $K_n$, 
$ (K_n)_{\l}([p^n]^{-1}P_i) = (K_n)_{\l} $ and hence 
\[
(L_n)_{\L} = (K_n)_{\l}
\]
for any $\L\mid \l$ in $L_n$. 
In particular, 
$L_n/K_n$ is unramified at 
all primes $\L\mid \l$ in $L_n$. 
As the extension $K_n/\Q$ is Galois, 
this extension $L_n/K_n$ is unramified above $l$. 
Hence $I_l = \set{1}$. 
\end{proof}

\subsection*{Main theorem}
In the rest of this section, we show the following theorem: 
 
\begin{theorem}
\label{thm:main}
For a prime $p>2$, and an elliptic curve $E$ over $\Q$ 
with minimal discriminant $\Delta$, 
put $K_n := \Q(E[p^n])$ and 
\[
	\quad \nu_l :=
	\begin{cases}
		\min\set{\ord_p(\ord_l(\Delta)),n},\ \mbox{if $E$ has split multiplicative reduction at $l$},\\
		n,\  \mbox{if $p=3$, $E$ has additive reduction at $l$, and $c_l = 3$},\\ 
		0,\  \mbox{otherwise},
	\end{cases}
\]
where $c_l$ is the Tamagawa number at $l$ $($\Cf \eqref{def:Tam}$)$.
We assume the following conditions:  
\begin{itemize}
	\item[\Full] $\Gal(K_1/\Q) \simeq GL_2(\Z/p\Z)$, and 
	\item[\Tor] $E(\Qp)[p] = 0$. 
\end{itemize}
Then, for all $n\in \Z_{\ge 1}$, 
the exponent $\kappa_n$  of $\#\Cl_p(K_n) = p^{\kappa_n}$ 
satisfies the following inequality: 
\[
	\kappa_n \ge 2n(r - 1) - 2 \sum_{l\neq p,\, l\mid \Delta} \nu_l, 
\] 
where $r$ is the rank of $E(\Q)$.
\end{theorem}

\begin{proof}
As in the beginning of this section, first we choose 
\begin{itemize}
	\item $P_1,\ldots , P_r \in E(\Q)$ : generators of the free part of $E(\Q)$, and put
\item $L_n := K_n\!\left([p^n]^{-1}P_1,\ldots, [p^n]^{-1}P_r\right)$. 
\end{itemize}
Next, we define 

\begin{itemize}
	\item $\wt{K}_n$\,: the Hilbert $p$-class field, that is, 
the maximal unramified abelian $p$-extension of $K_n$, and 
	\item $I := \Braket{I_{l}\, ; l = p \mbox{ or } l\mid \Delta}	\subset \Gal(L_n/K_n)$\,: 
	the subgroup generated by the inertia subgroups $I_p$ and $I_l$ for all prime number $l\mid \Delta$. 
\end{itemize}

\noindent
By class field theory (\Cf \cite{N},Chap.~VI, Prop.~6.9), we have 
\begin{equation}
\label{eq:An}
		\# \Cl_p(K_n) = [\wt{K}_n:K_n] \ge [L_n \cap \wt{K}_n:K_n] = \frac{[L_n:K_n]}{[L_n:L_n\cap \wt{K}_n]}.
\end{equation}
From the condition \Full\ and $p>2$, 
$\Phi:\Gal(L_n/K_n) \to E[p^n]^{\oplus r}$ defined in \eqref{eq:Phi} is bijective (\cite{SY15}, Thm.~2.4\footnote{
In \cite{SY}, it is considered the case where $p\ge 11$. 
However, the proof of Theorem 2.4 in \cite{SY} works 
for $p>2$. 
%
}, see also \cite{L}, Chap.~V, Lem.~1) 
and hence 
\begin{equation}
	\label{eq:deg}
 	[L_n:K_n] = p^{2nr}.
\end{equation}
Since the extension $L_n/K_n$ is unramified outside $\set{p,\infty} \cup \set{l\mid \Delta}$ (\cite{106}, Chap.~VIII, Prop.~1.5 (b)), we have
\begin{equation}
\label{eq:I}
	[L_n:L_n\cap \wt{K}_n] = [L_n:L_n^I] = \# I.
\end{equation}
Using the upper bound of $\#I_l$ given in Lemma \ref{lem:local} (for $l=p$ under the condition \Tor) and Lemma \ref{lem:local_l} (for $l\neq p$), we have 
\begin{equation}
	\label{eq:I3}
	\# I 
		\le \# I_p \cdot \prod_{l\neq p,\,l\mid \Delta} \# I_l 
		\le p^{2n}\cdot p^{2 \sum_{l\neq p,\, l\mid \Delta}\nu_l}. 
\end{equation}
Finally, 
Theorem~\ref{thm:main} follows from the following inequalities: 
\begin{align*}
	\# \Cl_p(K_n) & \ge\frac{[L_n:K_n]}{[L_n:L_n\cap \wt{K}_n]} \quad(\mbox{by \eqref{eq:An}})\\
	&=\frac{p^{2nr}}{\#I} \quad (\mbox{by \eqref{eq:deg} and \eqref{eq:I}})\\
	&\ge p^{2n(r - 1) - 2 \sum_{l\neq p,\, l\mid \Delta } \nu_l} \quad(\mbox{by \eqref{eq:I3}} ).
\end{align*}
\end{proof}

\providecommand{\bysame}{\leavevmode\hbox to3em{\hrulefill}\thinspace}
\providecommand{\href}[2]{#2}


\bigskip\noindent
Toshiro Hiranouchi 

\noindent
Department of Basic Sciences, 
Graduate School of Engineering, 
Kyushu Institute of Technology
1-1 Sensui-cho, Tobata-ku, Kitakyushu-shi, 
Fukuoka, 804-8550, JAPAN

\noindent
Email address: \texttt{hira@mns.kyutech.ac.jp}

\end{document}